\documentclass[11pt,reqno]{amsart}
\usepackage{amsmath, amsthm, amsfonts, amssymb,color,url}

\usepackage{tikz,amssymb,float}
\usetikzlibrary{arrows.meta,calc}
\usepackage{graphicx,paralist,epstopdf}
\makeatletter
\newcommand*\bigcdot{\mathpalette\bigcdot@{.5}}
\newcommand*\bigcdot@[2]{\mathbin{\vcenter{\hbox{\scalebox{#2}{$\m@th#1\bullet$}}}}}
\makeatother
\usepackage{wrapfig,hyperref}
\usepackage{mathtools}
\usepackage{fullpage}
\usepackage{subfig}
\usepackage{color}
\usepackage[normalem]{ulem}
\usepackage{tikz}
\usepackage{hyperref}
\hypersetup{
	colorlinks=true,
	linkcolor=blue,
	filecolor=magenta,
	urlcolor=cyan,
}
\theoremstyle{plain}
\newtheorem{theorem}{Theorem}[section]
\newtheorem{lemma}[theorem]{Lemma}
\newtheorem{cor}[theorem]{Corollary}
\theoremstyle{definition}
\newtheorem{definition}{Definition}
\newtheorem{problem}{Problem}

\newtheorem{remark}{Remark}

\newcommand{\M}{\operatorname{M}}

\title{One and Seven-Eighths Divisibility Problems of Propp}
\author[A.~Dunkelberg]{Aidan ~Dunkelberg}
\address{
Department of Mathematics \\
Indiana University \\
Bloomington, Indiana 47405} 
\email{aidunk@iu.edu}
\date{}

\begin{document}
	
\begin{abstract}
	We provide one full solution, and another nearly complete solution, to two problems from Jim Propp's 1999 article "Enumeration of matchings; problems and progress" (namely, Problems 30 and 31) involving divisibility properties of the number of matchings of two non-bipartite triangular graphs. We extend our method of proof of the second problem to show that the number of matchings of any graph composed of tetrahedral cells, such that the number of cells is suitably few relative to the number of vertices, is divisible by a power of 3; in particular, we then exhibit a collection of non-planar graphs whose number of matchings is divisible by 3.
\end{abstract}

\maketitle

\section{Introduction}

In 1996, shortly after the upsurge of research interest into exact enumeration of perfect matchings of graphs\footnote{A perfect matching (often simply called a matching) of a graph is a collection of disjoint edges that collectively contain all the vertices. We denote by $\M(G)$ the number of perfect matchings of the graph $G$. If $G$ is a weighted graph, the weight of a matching $\mu$ of $G$ is the product of the weights of the edges in $\mu$, and $\M(G)$ denotes the sum of the weights of all its matchings; the latter is called the matching generating function of $G$.} triggered by the solution of the Aztec diamond problem in \cite{eklp}, Jim Propp posed a list of twenty open problems in the enumeration of matchings in a MSRI lecture. Almost all of these problems, with one major exception in problem 18, dealt with bipartite graphs.

By 1999, more than half of the problems on Propp's original list had been solved, and he published the updated list \cite{proppprobs}, which, in addition to updating the status of the previous problems, contained twelve new problems (bringing the total number of problems to 32). Most of these problems concern nonbipartite graphs, to which less of the existing machinery of the subject can be applied. To date, most of the new problems from Propp's 1999 list have still not been solved. In this paper we treat two of these: Problems 30 and 31.

In Section 2, we address Propp's problem 30, originally conjectured by Matt Blum:

\begin{problem}
	Show that for the isoceles right triangle with extra edges, the number of matchings is always a multiple of 3. Furthermore, show that the exact power of 2 dividing the number of matchings is $2^{n/4}$ when $n$ is 0 modulo 4, and $2^0(=1)$ when $n$ is 3 modulo 4.
\end{problem}

	\begin{figure}[ht]
		\includegraphics[scale=1]{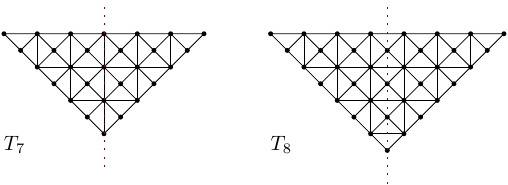}
		\caption{The isoceles right triangle with extra edges $T_n$.}
	\end{figure}
	
We will show first, by appealing to the horizontal symmetry of the graph, that 2 does not divide $\M(T_n)$ when $n$ is 3 modulo 4, and that $2^{n/4}$ divides $\M(T_n)$ when $n$ is 0 modulo 4. (We are, however, unable to show that no greater power of 2 divides $\M(T_n)$ in the latter case; this is the ``missing one-eighth'' referenced in the title of the paper.)
We next find a collection of subgraphs of $T_n$ whose number of matchings is divisible by 3 and use them to show that 3 divides the number of matchings of $T_n$.

We then devote Section 3 to a proof from scratch of Propp's next problem, problem 31:
	
	\begin{problem}
		Show that for the equilateral triangle graph with extra vertices
		and edges, the number of matchings is always a multiple of 3.
	\end{problem}
	
	\begin{figure}[ht]
		\includegraphics[scale=0.8]{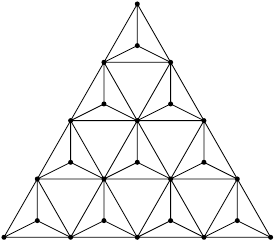}
		\caption{The equilateral triangle graph $E_n$ with extra vertices and edges.}
	\end{figure}
	
In fact, we show that the number of matchings is divisible by $3^{n/2}$. Our proof will involve a local replacement procedure, similar to the Spider Lemma introduced in \cite{gendom}, for each triangle with midpoint of the above graph.

Finally, in Section 4, we will extend the ideas of the previous section to graphs composed of tetrahedral cells, yielding a highly general class of graphs whose number of matchings is divisible by 3. In particular, we demonstrate some non-planar graphs which we show fall into the previous class, adding to a very limited body of knowledge concerning the enumeration of matchings of non-planar graphs.
	
	\section{Problem 30}
	
	Problem 30 has two halves: first concerning divisibility by 3, then concerning the exact power of 2 that divides the number of matchings. The latter half, regarding the power of 2, is further split into two cases, according to $n$ being congruent to 0 or 3 modulo 4; in the former case, the statement that $2^{n/4}$ is the exact power dividing the number of matchings is implicitly two statements, that $2^{n/4}$ does divide the number of matchings, and that no higher power does so. In this section we prove all these statements except the one saying that when $n$ is a multiple of 4, no power of 2 greater than $2^{n/4}$ divides $\M(E_n)$. Those concerning the power of 2 form Theorem \ref{two}, while Theorem \ref{three} covers divisibility by 3.
	
	\begin{theorem}\label{two}
		Let $T_n$ be the isoceles right triangle graph with extra edges shown in Figure \ref{isotri}, where $n$ is the number of vertices along a side of the triangle. Then $\M(T_n)$ is divisible by $2^{n/4}$ when $n \equiv 0 \bmod 4$ and is not divisible by $2$ when $n \equiv 3 \bmod 4$.
	\end{theorem}
	
	\begin{figure}[ht]
		\includegraphics[scale=1.4]{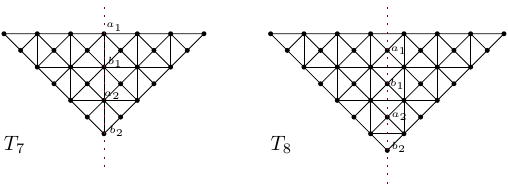}
		\caption{$T_n$ for $n\equiv3 \bmod 4$ and $n\equiv 0 \bmod 4$.}\label{isotri}	\end{figure}
	
	The proof employs the following useful lemma proven by Ciucu in \cite{fact}, which we restate below without proof. Given a symmetric graph $G$, denote by $a_1, b_1, a_2, b_2, \ldots, a_{w(G)}, b_{w(G)}$ the vertices along its symmetry axis, where the width $w(G)$ is equal to half the number of vertices along its symmetry axis. A \textit{reduced} subgraph of $G$ is obtained by removing all of the edges incident to each $a_i$ either above or below the symmetry axis.
	
	\begin{lemma}[Ciucu, \cite{fact}]\label{reduced}
		All $2^{w(G)}$ reduced subgraphs of a symmetric weighted graph $G$ have the same matching generating function.
	\end{lemma}
	
\begin{proof}[Proof of Theorem \ref{two}]
	We first consider the case when $n \equiv 0 \bmod 4$. In this case, $T_n$ is symmetric about a vertical line in the orientation of Figure \ref{isotri}; label the vertices on this line $a_1, b_1, \ldots, a_{n/4}, b_{n/4}$ from top to bottom. Since no edges of $T_n$ lie on the symmetry axis when $n \equiv 0 \bmod 4$, every matching of $T_n$ must match each $a_i$ either to the left or the right of the symmetry axis. This means that each matching of $T_n$ is a matching of exactly one reduced subgraph of $T_n$. In particular, this gives a partition of the set of matchings of $T_n$ into $2^{n/4}$ classes, each of which contains the matchings which restrict to perfect matchings of a particular reduced subgraph. Since by Lemma \ref{reduced}, all reduced subgraphs have the same number of matchings,
	\begin{equation}\label{teven} \M(T_n) = \sum_{G \text{ a reduced subgraph of }T_n} \M(G) = 2^{n/4} \M(R)\end{equation}
	where $R$ is any fixed reduced subgraph of $T_n$. Thus $\M(T_n)$ is divisible by $2^{n/4}$.
	
	Now consider the case when $n \equiv 3 \bmod 4$. We partition the matchings of $T_n$ into those which are symmetric about the horizontal symmetry axis and those which are not. Since reflection about the symmetry axis is an involution on the set of matchings, and any matching fixed by this involution is symmetric, there must be an even number of matchings which are not symmetric.
	
	On the other hand, to form a symmetric matching of $T_n$, we must match the symmetry axis internally, which forces us to use the edges $\{ a_1b_1, a_2b_2, \ldots, a_nb_n\}$. In such a matching, the subgraph to the right of the symmetry axis will be matched internally. However, this subgraph has only one perfect matching, which contains exactly those edges which are northeast of a vertex of degree 4. We show this as follows: note that the bottom-most vertex to the right of the symmetry axis has only one neighbor to its right, so it is forced to be matched northeast. All vertices in this column are subsequently forced in the same way. The same argument applies to each of the remaining columns in succession.

		\begin{figure}[ht]
			\includegraphics[scale=1.6]{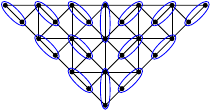}
			\caption{The unique symmetric matching of $T_7$.}\label{symmmatching}
		\end{figure}
		
		By symmetry, the same observation holds for the subgraph to the left of the symmetry axis; thus, there is exactly one symmetric matching of $T_n$. It follows that, for $n \equiv 3 \bmod 4$, $T_n$ has an odd number of matchings.
	\end{proof}

\begin{remark}
As mentioned previously, Problem 30 additionally conjectures that the power of 2 dividing $\M(T_n)$ is \textit{exactly} $n/4$ when $n \equiv 0 \bmod 4$. We have not proven this part of the conjecture, but we have substantial numerical data (up to $n=100$) which validates it. Further, Equation \ref{teven} indicates that demonstrating the existence of any reduced subgraph of $T_n$ ($n \equiv 0  \bmod 4$) with an odd number of matchings would be enough to prove this missing part of Problem 30.
\end{remark}

Next, we address the first half of Problem 30, namely the statement that $\M(T_n)$ is divisible by 3. To do so, we first compute exactly the matching number of the following family of subgraphs of $T_n$.

Define $V_n$ for $n \geq 1$ to be the graph shown in Figure \ref{vees}, which is the restriction of $T_{n+2}$ to its top three rows of vertices. Further, let the vertices on the bottom row be labeled, from left to right, $\{x_1, \ldots, x_n\}$.  
	
	\begin{figure}[ht]
		\includegraphics[scale=1]{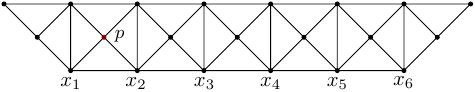}
		\caption{The graph $V_n$ for $n=6$.}\label{vees}
	\end{figure}
	
\begin{lemma}
	If $S \subset \{x_1, \ldots, x_n\}$ has cardinality $s$, with $n-s$ odd, then
	\[ \M( V_n \setminus S) = 3 \cdot 6^{\frac{n-s-1}2}.\]
\end{lemma}

As $V_n$ has $3n+3$ vertices, in order for $V_n \setminus S$ to be a matchable graph, we must have $3n+3-s \equiv n+1-s \equiv 0 \bmod 2$, hence the stipulation that $n-s$ be odd.

\begin{proof}
	Consider the case when $S=\emptyset$ (and $n$ is odd). In this case, we may partition the matchings of $V_n$ by the edge they contain incident to the second vertex in the middle row (which we label $p$), as shown in Figure \ref{veespart}. In all cases, after removing $p$ and the vertex it is matched to, the remaining edge above or below $p$ is a bridge, meaning that removing it would disconnect the graph.  
		
		\begin{figure}[ht]
			\includegraphics[scale=1]{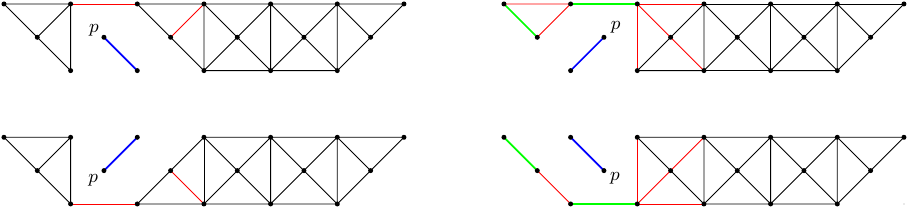}
			\caption{Partitioning to create a recurrence for $V_n$. In each diagram, we force the blue edge incident to $p$ to be included in our matching; other forced edges are colored green, while edges which cannot be used in a matching are colored red.}\label{veespart}
		\end{figure}
		
	For the two cases on the left of Figure \ref{veespart}, where $p$ is matched to the right, this bridge cannot be contained in a matching, as using it would disconnect the graph into two components with an odd number of vertices; thus, we may remove it. Then, the component to the right of $p$ has $\M(V_{n-2})$ matchings in both cases. Indeed, when $p$ is matched down and to the right, the component to the right of $p$ is exactly $V_{n-2}$, since we have removed the two leftmost vertices of each row of $V_n$.
		
		\begin{figure}[ht]
			\includegraphics[scale=1]{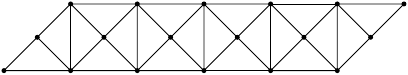}
			\caption{The graph $H_n$ for $n=5$.}\label{aitch}
		\end{figure}
		
		If $p$ is matched up and to the right, the component to the right of $p$ is similar to $V_{n-2}$, but the leftmost vertex has a neighbor on the bottom row and no neighbors on the top (see Figure \ref{aitch}); we call this graph $H_{n-2}$. However, for both the graphs $H_n$ and $V_n$, the leftmost vertex in the middle row has one edge that does not border the infinite face (colored red in Figure \ref{veespart}). In both cases, this edge cannot be used in a matching, as the leftmost vertex of the graph must be matched to one of its two endpoints; thus, we may remove it without changing the number of matchings. After removing this edge, the resulting graphs are isomorphic, so
		\begin{equation} \label{heqv}
			\M(H_n) = \M(V_n).
		\end{equation}
		To conclude the analysis of these two cases, we see that in each case where $p$ is matched to the right, there are two ways to match the four-vertex component to the left of $p$. Thus, each of these cases contributes $2\M(V_{n-2})$ matchings.
		
		For the two cases when $p$ is matched to the left (shown on the right in Figure \ref{veespart}), we again have a bridge above (resp., below) $p$. However, this time there are an odd number of vertices remaining on either side of $p$, so the bridge must be included in any matching of the graph. In each case, there will be one remaining edge to the left of $p$, which also must be included in any matching, and the component to the right of $p$ is either $V_{n-2}$ or $H_{n-2}$, so has $\M(V_{n-2})$ matchings. So each of these cases contributes $\M(V_{n-2})$ matchings.
		It follows that
		\[ \M(V_n) = 2\M(V_{n-2}) + 2\M(V_{n-2}) + \M(V_{n-2}) + \M(V_{n-2}) = 6\M(V_{n-2})\]
		and since $\M(V_1)=3$ (as can be readily checked), the formula holds when $s=0$.

		For the case when $S \neq \emptyset$, we proceed by induction on $n$. Let $x_i$ be the rightmost vertex contained in $S$. If $x_{i-1} \in S$ also, denote by $v_i$ the vertex above $x_i$, $v_{i-1}$ the vertex above $x_{i-1}$, and $u$ the vertex in the middle of the square formed by these four vertices; this labeling is shown on the left of Figure \ref{deletedverts}. Since using the edge between $v_i$ and $v_{i-1}$ in a matching would isolate $u$, we may delete that edge. Then, the remaining path of length 2 between $v_i$ and $v_{i-1}$ may, by the vertex splitting lemma, be replaced by a single vertex (which we label $v_{i-1}'$) without changing the number of matchings. The resulting graph, as seen in Figure \ref{deletedverts}, is $V_{n-1}$ with all vertices of $S$ removed except for $x_i$. Thus
		\[ \M(V_n \setminus S) = \M(V_{n-1} \setminus \{S\setminus x_i\}) = 3 \cdot 6 ^{(n-1)-(s-1)-1} = 3 \cdot 6^{n-s-1} \]
		by the inductive hypothesis.
		
		\begin{figure}[ht]
			\includegraphics[scale=1]{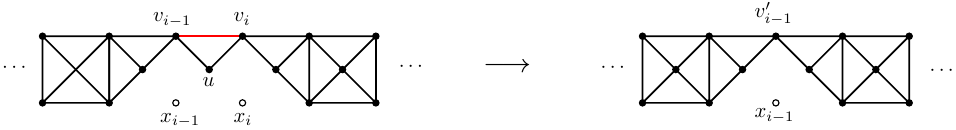}
			\caption{The case $x_{i-1} \in S$. The red edge cannot be included in a matching.}\label{deletedverts}
		\end{figure}

		For the remaining case, suppose $x_{i-1} \notin S$. In this case, $v_i$ is a cut vertex for the graph, so it must be matched to the side with an odd number of vertices. Suppose that $v_i$ must be matched to the right (the other case can be handled similarly). Then we may remove the edges to the left of $v_i$, disconnecting the graph as shown in Figure \ref{deletedvert}.
		
		\begin{figure}[ht]
			\includegraphics[scale=1]{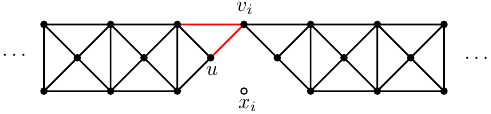}
			\caption{The case $x_{i-1} \notin S$. Red edges cannot be included in a matching.}\label{deletedvert}
		\end{figure}
		
		Now, the rightmost vertex of the component to the left of $v_i$, labeled by $u$ in Figure \ref{deletedvert}, must either be matched down and to the left or up and to the left. In the first case, the rest of the component is $V_{i-1} \setminus (S \setminus x_i)$, and in the second it is $H_{i-1} \setminus (S \setminus x_i)$. However, the argument leading up to equation \ref{heqv} only depends on the end of the graph, so still holds if some vertices of the bottom row are removed. Thus, the component to the left of $v_i$ has $2\M(V_{i-1} \setminus (S \setminus x_i))$ matchings.
		Since the component to the right has $\M(V_{n-i})$ matchings, by choice of $x_i$,
		\[ \M(V_n \setminus S) = 2\M(V_{i-1} \setminus (S \setminus \{x_i\})) \M(V_{n-i}) = 2 \cdot 3 \cdot 6^{i-1 - (s-1)-1} \cdot 3 \cdot 6^{n-i-1} = 3 \cdot 6^{n-s-1}\]
		again by the inductive hypothesis. The result follows.
		\end{proof}
		
		We are now ready to prove the second part of Problem 30.
		
		\begin{theorem}\label{three}
		$\M(T_n)$ is divisible by 3 for all $n$.
		\end{theorem}
		
		\begin{proof}
		Consider the subgraph of $T_n$ consisting of the top three rows (in the orientation of Figure \ref{isotri}), which is $V_{n-2}$; label the third row of $T_n$ by $x_1, \ldots, x_{n-2}$ as in the lemma above. A matching of $T_n$ may be constructed by selecting which elements of $\{x_1, \ldots, x_{n-2}\}$ are matched within $V_{n-2}$ and which are matched below: hence
		\[ \M(T_n) = \sum_{S \subset \{x_1, \ldots, x_{n-2}\}} \M(V_{n-2} \setminus S) M\big(T_n \setminus (V_{n-2} \setminus S)\big).\]
		Since by the preceding lemma, $\M(V_n \setminus S)$ is divisible by 3 for all sets $S$, the statement follows.
		\end{proof}
		
		We observe that only $V_{n-2}$ plays a role in proving this divisibility by 3; the structure of the rest of $T_n$ does not contribute to the proof. This implies that a more general theorem is true for graphs that have a subgraph $V_n$ ``attached'' to the rest of the graph at the $x_i$'s. We state this more formally as the following corollary.
		
		\begin{cor}
		Let $G$ be any finite graph, and suppose that there exists an isomorphism to a subgraph $H$ of $G$, denoted $\phi : V_n \xrightarrow{\cong} H$, with the property that whenever $v \notin H$, $w \in H$, and $vw \in E(G)$, then $w = \phi(x_i)$ for some $i$. Then $3$ divides $\M(G)$.
		\end{cor}
		
		\begin{proof}
		Using the setup of the previous theorem, we write
		\[ \M(G) = \sum_{S \subset \{\phi(x_1), \ldots, \phi(x_n)\}} \M(H \setminus S) \M(G \setminus (H \setminus S))\]
		and by the isomorphism, $\M(H \setminus S)$ is again divisible by 3 for all sets $S$.
		\end{proof}

		\section{Problem 31}
		
		In this section, we will prove a strengthening of Problem 31 of \cite{proppprobs}, which claimed that the number of matchings of the equilateral triangle graph with included midpoints $E_n$, described in Figure \ref{equil}, is divisible by 3. We begin, however, with the following necessary lemma, which is similar to the  well-known result known variously as the urban renewal or spider lemma (presented, for example, in section 5 of \cite{gendom}), but having a subunit of $E_n$ in place of the 4-cycle of urban renewal.
	
	\begin{figure}[ht]
		\includegraphics[scale=0.9]{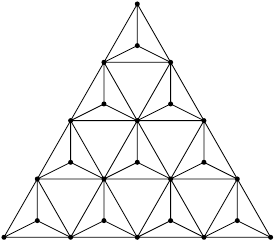}
		\caption{The equilateral triangle graph $E_n$ for $n=5$.} \label{equil}
	\end{figure}
	
\begin{lemma}\label{equillemma}
	Let $G$ be a graph which contains the local configuration on the left of Figure \ref{trianglesplit}, where the vertices of the inner triangle with midpoint have no neighbors besides the ones indicated. 
	Replacing this configuration by the local configuration on the right of Figure \ref{trianglesplit}, i.e., removing the vertices of the inner triangle with midpoint and all shown edges, and adding edges of weight $\frac 1 3$ between the outer vertices, results in a graph $G'$ with $\M(G') = \frac 1 3 \M(G)$.
\end{lemma}

	\begin{figure}[ht]
		\includegraphics[scale=0.4]{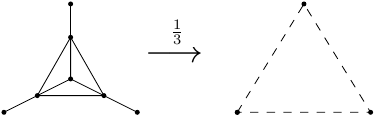}
		\caption{Local replacement for triangles with midpoints. All vertices in the triangle on the left have no other neighbors besides the ones shown. Dashed edges have weight $\frac 1 3$.}\label{trianglesplit}
	\end{figure}

\begin{proof}
	First, we observe that both before and after the local replacement, an odd number of vertices must be matched outside of the local configuration. 
	
	If all three of the outer vertices are matched outside, as in the upper case of Figure \ref{equilspider}, there are three ways to match the four vertices of the triangle with midpoint in the original configuration, each of which corresponds to the same matching in the replacement configuration. If only one outer vertex is matched outside (the lower case of Figure \ref{equilspider}), we must use the edges between the other two outer vertices and the corresponding exterior vertices of the triangle with midpoint, meaning that the midpoint of the triangle must be matched to the remaining exterior vertex. Thus, there is only one matching of the original configuration. For the replacement configuration, there is also only one matching, but this time, we must use one edge of the triangle to match the two vertices that are not matched to the outside, so our matching has weight $\frac 1 3$.
	
	Thus, in both cases, any matching of the rest of the graphs $G$ and $G'$ (which are the same outside of the local configuration) contributes $\frac 1 3$ as much to the total weight of matchings of $G'$ as of $G$.
	\end{proof}
	
	\begin{figure}[ht]
		\includegraphics{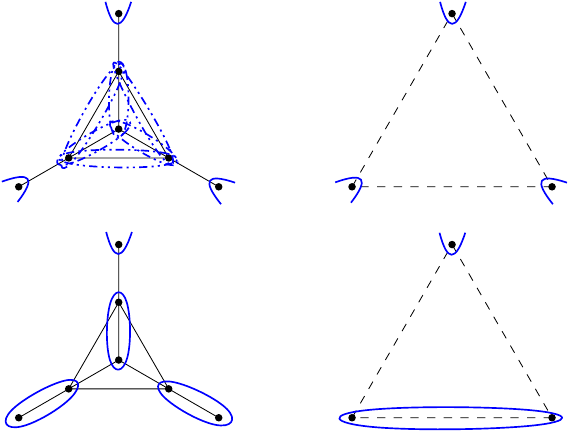}
		\caption{The correspondence between matchings before and after local replacement; dot-dashed circles represent possible matchings, while an arc around a vertex means it is matched outside the depicted region. The bottom pair may also be rotated by $120^\circ$ or $240^\circ$ to obtain an analogous correspondence.}\label{equilspider}
	\end{figure}
	
\begin{remark}
	The preceding result is also a corollary of a result that we will prove later, Lemma \ref{tetraspider}; in particular, Lemma \ref{tetraspider} simplifies to the preceding result when one outer vertex has exactly one leaf as a neighbor.
\end{remark}

Now we are ready to prove the main theorem of this section, a strengthening of the statement in Propp's Problem 31.

\begin{theorem}\label{equilthree}
	Let $E_n$ be the equilateral triangle graph with included midpoints shown in Figure \ref{equil}, where $n$ is the number of vertices on the bottom row. Then $\M(E_n)$ is divisible by $3^{n/2}$.
\end{theorem}

\begin{remark}
	Unlike in the previous section, we know from computations that $3^{n/2}$ is not always the exact power of 3 dividing $\M(E_n)$, so we cannot formulate a conjecture similar to the unsolved portion of Problem 30. Nevertheless, we have found many cases where $3^{n/2}$ is the exact power dividing $\M(E_n)$, including as large as $n=30$, so we think it likely that the above is the strongest general statement about the power of 3 that can be proven.
\end{remark}

\begin{proof}
	When $n$ is odd, $E_n$ has an odd number of vertices, implying $\M(E_n)=0$, at which point the theorem holds trivially. In what follows we assume, often tacitly, that $n$ is even.
	
	Call the vertices in the center of a triangle with midpoint of $E_n$ interior vertices, and the others exterior vertices. Perform vertex splitting (as in e.g. Lemma 1.3 of \cite{fact}) around each exterior vertex $v$ once for each triangle with midpoint incident to $v$, so that every exterior vertex is replaced by a star joined to $\deg(v)/3$ triangles by paths of length 2, as shown in Figure \ref{equilsplit}. This results in a new graph $E_n'$ with $\M(E_n')=\M(E_n)$.
		
		\begin{figure}[ht]
			\includegraphics[scale=0.75]{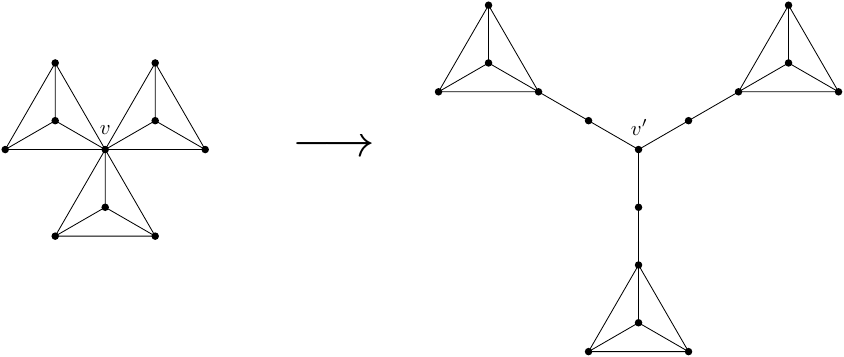}
			\caption{The vertex splitting process used to obtain $E_n'$.}\label{equilsplit}
		\end{figure}
		
Now, the exterior vertices of each triangle of $E_n'$ have only one other neighbor besides those in the triangle. Thus, around each triangle we have the local configuration on the left of Figure \ref{trianglesplit}, so we may apply Lemma \ref{equillemma} successively around each of the $\binom{n}{2}$ triangles with midpoint of $E_n'$. This yields a new graph $E_n''$ satisfying
\[ \M(E_n'') = \left( \frac 1 3 \right)^{\binom{n}{2}} \M(E_n).\]
The edges of $E_n''$ arising from the local replacement procedure have weight $\frac 1 3$. The others, which all have weight 1, are exactly those which connect a vertex that is involved in the local replacement procedure to one that is not. But any vertex not involved in the local replacement (shown as green vertices in Figure \ref{equil2}) was at the center of a vertex splitting procedure, so there are as many such vertices as exterior vertices of $E_n$, that is, $\binom{n+1}2$. Since exactly one edge of the second class must be used for each of these vertices, we may replace all these edges by edges of weight $\frac 1 3$ to obtain a new graph $E_n'''$ with
\[ \M(E_n''') = \left( \frac 1 3 \right)^{\binom{n+1} 2} \M(E_n'') = \left( \frac 1 3 \right)^{\binom{n+1}{2} +  \binom{n}{2}} \M(E_n).\]

		\begin{figure}
			\includegraphics[scale=0.8]{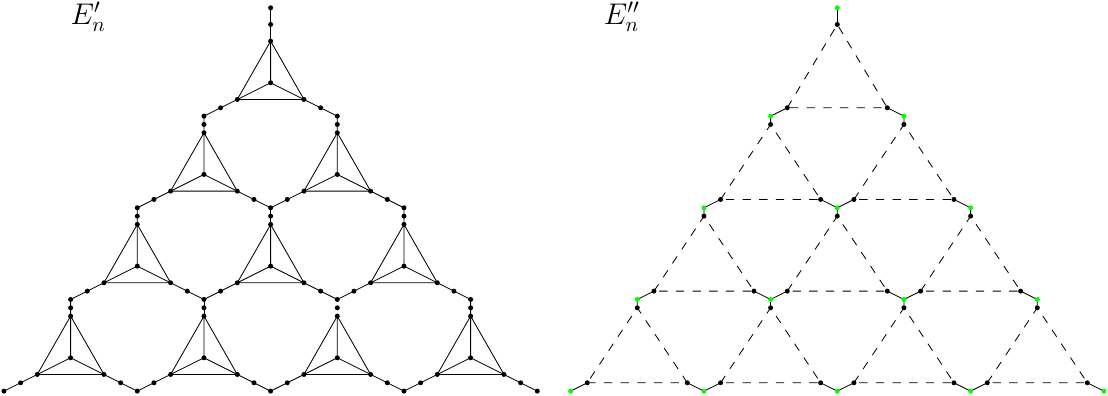}
			\caption{$E_n'$, obtained from $E_n$ by vertex splitting, and $E_n''$, obtained from $E_n'$ by local replacement. Dashed edges have weight $\frac 1 3$; note that edges of weight 1 are incident to exactly one green vertex.}\label{equil2}
		\end{figure}
		
Now, we see that $E_n'''$ has $3\binom{n}{2}+\binom{n+1}{2}$ vertices—three vertices for every triangle of $E_n$, and one for every vertex of $E_n$—so every matching of $E_n'''$ has $\frac 3 2\binom{n}{2}+ \frac 1 2 \binom{n+1}2 $ edges. Thus, changing all edges of $E_n$ to have weight 1 increases the number of matchings by $3^{\frac 3 2\binom{n}2+ \frac 1 2 \binom{n+1}2}$. It follows that
\[ \M(E_n''',\text{all weights 1}) = 3^{\frac 3 2\binom{n}2+ \frac 1 2 \binom{n+1}2}\M(E_n''') = 3^{\frac 3 2\binom{n}{2}+ \frac 1 2 \binom{n+1}2  - \binom{n}2-\binom{n+1}2}\M(E_n) \]
so, rearranging,
\[ \M(E_n) = 3^{\frac 1 2 \left(\binom{n+1}2-\binom{n}2\right)} \M(E_n''', \text{all weights }1)\]
and since the number of matchings of any graph with all edge weights 1 is an integer, and $\binom{n+1}2 - \binom n 2 = n$, it follows that
\[3^{n/2} \mid \M(E_n). \]
\end{proof}

\section{A generalization and a non-planar divisibility result}
We will now extend the results of the previous section to a more general case, which in particular yields divisibility results for subgraphs of a particular three-dimensional lattice.

As in the previous section, we begin with a result similar to the urban renewal or spider lemma (section 5 of \cite{gendom}); in this case, our result is simply a generalization of the existing spider lemma.

The result we present here was discovered independently by Ciucu in the mid-2010s {(private communication)}.

	\begin{figure}[ht]
		\includegraphics[scale=1]{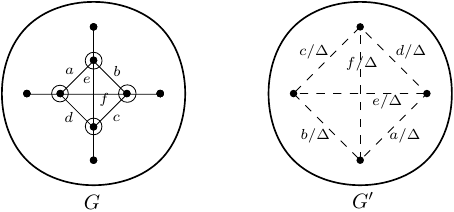}
		\caption{The circled vertices have no other neighbors besides the ones shown.} \label{spider1}
	\end{figure}
	
\begin{lemma}[Tetrahedral Spider Lemma]\label{tetraspider} Let $G$ be a graph which contains the local configuration shown on the left of Figure \ref{spider1},
	that is, which has a subgraph $K$ isomorphic to a weighted $K_4$ such that each vertex of $K$ is connected by an edge of weight 1 to a unique single additional neighbor. Let $G'$ be the graph obtained from $G$ by deleting the vertices of $K$ and including edges between its neighbors as shown on the right of Figure \ref{spider1}. Then
	\[ \M(G) = \Delta \cdot \M(G')\]
	where $\Delta = ac + bd + ef$ is the sum of the products of weights of opposite sides of $K$.
\end{lemma}

\begin{proof}
	Denote the set of external neighbors of vertices of $K$ by $N$. We will partition the matchings of $G$ by which of the vertices of $N$ are matched to the outside of the local configuration shown in Figure \ref{spider1}. Each of the sets of matchings in the partition is readily identified with the set of matchings of $G'$ with the corresponding vertices matched to the outside of the local configuration of $G'$. Up to rotation, the four such identifications are shown in Figure \ref{spider2}. We now calculate explicitly that for each identification, the total weight of the matchings of the replacement is $\frac 1 \Delta$ that of the original configuration.
	
	When all four vertices of $N$ are matched inward, the edges of the original configuration contribute weight 1 to each matching. On the other hand, there are three ways to match the replacement configuration (as shown in the top right of Figure \ref{spider2}), which contribute weight $\frac {ac}{\Delta^2}, \frac{bd}{\Delta^2}$, and $\frac{ef}{\Delta^2}$, respectively, to any matching of the rest of the graph. The total contribution is then $\frac{ac+bd+ef}{\Delta^2} = \frac 1 \Delta$.

		\begin{figure}
			\includegraphics[scale=1.2]{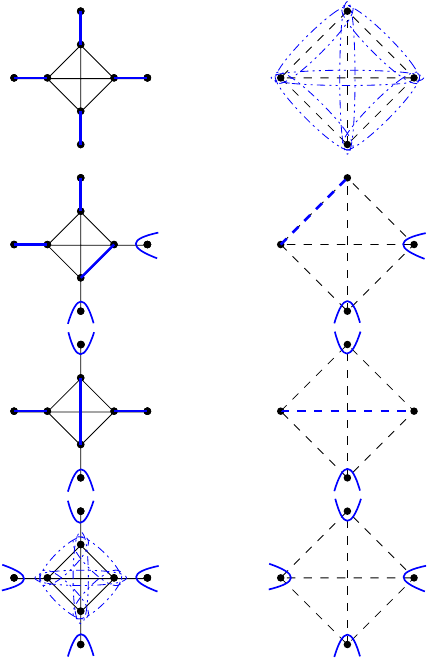}
			\caption{Four types of matching bijections under the local replacement. Dot-dashed circles represent possible matchings, whereas an arc around a vertex means it is matched outside the region.}\label{spider2}
		\end{figure}
		
When two vertices of $N$ are matched inwards, as in the second and third rows of Figure \ref{spider2}, there is one way to match the original configuration, which uses the edge of the $K_4$ connecting the two vertices of $K$ whose neighbors in $N$ are matched outwards. Likewise, there is one way to match the replacement configuration, which uses the edge of this $K_4$ whose vertices are not matched outwards. One may readily check that edges that are, in this sense, opposite between the original and replacement configuration have the same weight, up to modification by $\frac 1 \Delta$, so again the ratio of weights is $\frac 1 \Delta$ in all cases.

Finally, when no vertices of $N$ are matched inwards, the original configuration may be matched in three ways, with weights $ac, bd$, and $ef$ respectively, for a total weight of $\Delta$. On the other hand, there is only one way to match the replacement configuration, which uses only edges of weight 1; so the ratio of weights is again $\frac 1 \Delta$.
\end{proof}

In the spirit of the notion of a cellular graph as defined in \cite{cellular}, we present the following definition.

\begin{definition}
Define a \textit{tetracellular graph} to be a graph $G$ which can be covered by subgraphs (called cells), each isomorphic to $K_4$, such that every edge is contained in exactly one cell.
\end{definition}

\begin{theorem}\label{tetra}
Let $G$ be a tetracellular graph with $V$ vertices and $C$ cells. If $V > 2C$, then
\[ 3^{\frac{(V-2C)}2} \text{ divides } \M(G).\]
\end{theorem}

\begin{proof}
The result will follow from the following construction. We begin with a vertex splitting procedure similar to the splitting process with which we began Theorem \ref{equilthree}.

First, at every vertex $v$ of $G$ we perform vertex splitting (as before; again cf. Lemma 1.3 of \cite{fact}) once for each cell that contains $v$, separating the vertices in that cell from the other neighbors of $v$. This action on a single vertex is shown in Figure \ref{tetra1}; in words, whenever $n$ cells met at a vertex of $G$, the graph resulting from vertex splitting has $n$ paths of two edges all joined at a single point, with a single cell joined to their other ends. Applying this process to each vertex of $G$ produces a new graph $G'$ with $\M(G') = \M(G)$.
		\begin{figure}[ht]
			\includegraphics[scale=0.8]{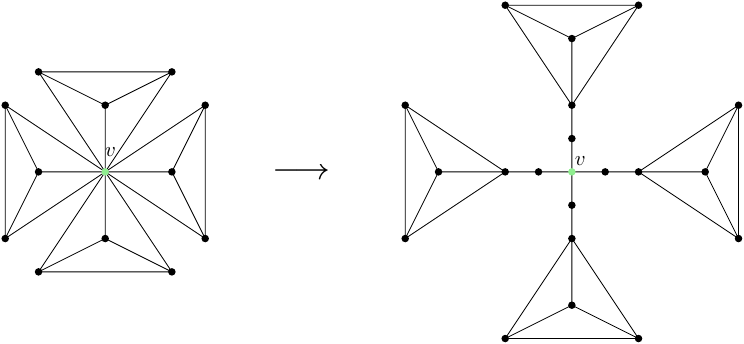}
			\caption{The vertex splitting process used in proving Theorem \ref{tetra}.}\label{tetra1}
		\end{figure}
		
The cells of $G$ are exactly the cells of $G'$, but each vertex of each cell of $G'$ was involved in a splitting process as described above, so by construction has only one neighbor outside of the cell. Thus, we can apply Lemma \ref{tetraspider} around each cell of $G'$, creating a new graph $G''$ (an example is shown in Figure \ref{tetra2}). 
Since each edge of $G'$ has weight 1, $\Delta=3$ for all cells, so
\[ \M(G) = \M(G') = 3^C \M(G'').\]
		
		\begin{figure}
			\includegraphics[scale=1]{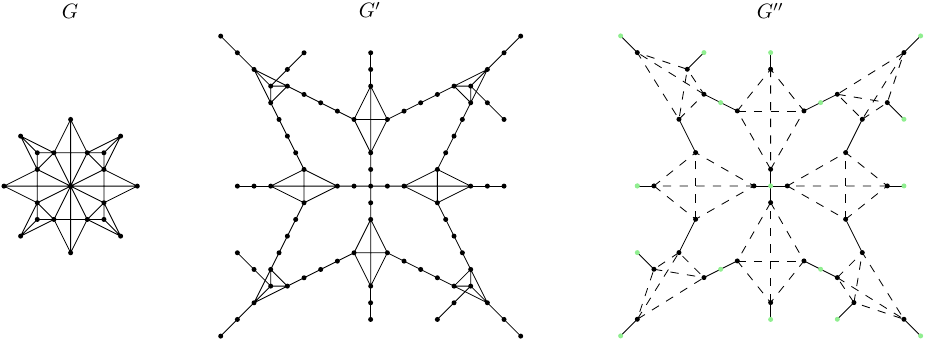}
			\caption{An example application of the process of Theorem \ref{tetra}; dashed edges have weight $\frac 1 3$. Each green vertex of $G''$ is incident only to edges of weight 1. $G'''$ is obtained by re-weighting these edges.} \label{tetra2}
		\end{figure}
		
Now, since every edge of $G$ is contained in a cell, every edge of $G'$ is either contained in a cell of $G'$ or is generated by the splitting process. Those generated by the splitting process are further partitioned into two types: those which are incident to a vertex at the center of a splitting process, colored green in Figures \ref{tetra1} and \ref{tetra2}, and those which are not. In $G''$, exactly those edges incident to a green vertex remain; the rest are removed in our applications of the tetrahedral spider lemma and replaced by edges of weight $\frac 1 3$. Let $G'''$ be the graph obtained by re-weighting all edges of $G''$ incident to a green vertex with weight $\frac 1 3$, so that all edges of $G'''$ have weight $\frac 1 3$. Since exactly one edge incident to each green vertex is used in any perfect matching of $G''$,
\[ \M(G) = 3^C \M(G'') = 3^{C+V} \M(G''').\]
As $G'''$ has $4C+V$ vertices, any matching of $G'''$ uses $2C+\frac{V}2$ edges, each of weight $\frac 1 3$. Thus $3^{2C+\frac{V}2}\M(G''')$ is an integer, so since
\[ \M(G) = 3^{C+V} \M(G''') = 3^{-C + \frac V 2} \left( 3^{2C + \frac V 2} \M(G''') \right),\]
$3^{\frac{V}2-C}$ divides $\M(G)$.
\end{proof}

\begin{remark}
If instead $G$ is a cellular graph (exactly as defined in \cite{cellular}) one may use the constructions of the proof of Theorem \ref{tetra}, combined with the traditional spider lemma, to show that $2^{\frac{V-2C}2}$ divides $\M(G)$. %In particular, this yields that the number of matchings of the Aztec diamond of order $n$, $AD_n$, is divisible by $2^n$. While this is a weaker version of the longstanding result that $\M(AD_n)=2^{\frac{n(n+1)}{2}}$, it may provide an illustration of the relative strength of the divisibility results that Theorem \ref{tetra} achieves.
\end{remark}

The immediate usefulness of Theorem \ref{tetra} may not be clear; after all, with the exception of the previously examined case of Problem 31, we know of no tetracellular graphs appearing in the literature on matchings.

However, we will show in the following construction that Lemma \ref{tetra} applies to certain non-planar tetracellular graphs. Even though matchings have been enumerated for certain non-planar graphs previously (e.g., in \cite{fisher}), results in this case\footnote{More precisely, the case of graphs which cannot be embedded in a 2-manifold.} have been much more difficult to obtain, so that we believe even a general divisibility result represents significant progress.

Let $L$ be the infinite lattice of regular tetrahedrons such that:
\begin{enumerate}[(a)]
\item Exactly two tetrahedrons meet at each vertex; and
\item At any such vertex, each incident edge from one tetrahedron is collinear with an incident edge from the other tetrahedron.
\end{enumerate}
Condition (b) may also be stated as: whenever two tetrahedrons meet at a vertex, their faces opposing that vertex are parallel to each other. A two-dimensional ``slice'' of $L$ is shown in Figure \ref{graphite}, and a three-dimensional view of $L$ is shown in Figure \ref{lattice}.
	
	\begin{figure}[ht]
		\includegraphics[scale=0.7]{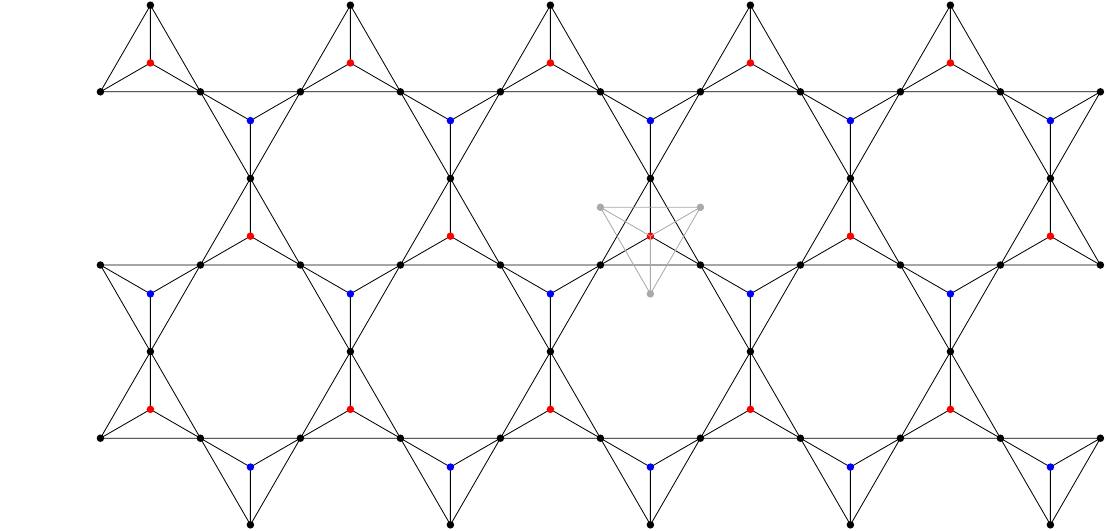}
		\caption{One layer of the lattice $L$; blue vertices lie above the plane, while red ones lie below it. Adjoining tetrahedrons on the layers above and below have a $60^\circ$ rotational offset, as exemplified by the gray tetrahedron in the center. }\label{graphite}
	\end{figure}
	
	\begin{cor}
		Let $G$ be any finite subgraph of $L$ containing only whole tetrahedrons. Then $\M(G)$ is divisible by $3$.
	\end{cor}
	
	\begin{figure}
		\includegraphics[scale=0.9]{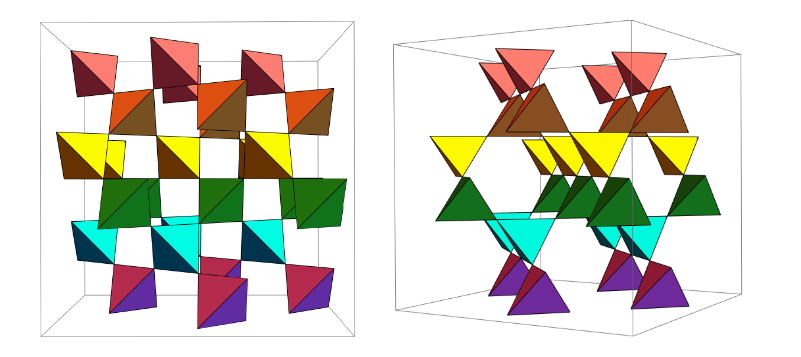}
		\caption{Two views of the lattice $L$, restricted to the cube $[-3,3]^3$.}\label{lattice}
	\end{figure}
	
\begin{proof}
	Since $G$ contains only whole tetrahedrons, it is tetracellular, so we may apply Lemma \ref{tetra} as long $G$ has more than twice as many vertices as cells. But each cell of $G$ contains four vertices, whereas each vertex of $G$ is contained in at most two cells. Set $d_V$ to be the average number of cells containing any vertex. We may count the number of pairs consisting of a cell and a vertex contained in that cell in two ways. First, we may choose a cell and then choose one of the four vertices inside it, which we can do in $4C$ ways; alternatively, we may choose a vertex and pick one of the cells it is contained in, which we may do in $V \cdot d_V$ ways. Since some vertex of $G$ must not be in two cells (as $G$ is finite), $d_V < 2$, so
	\[ 4C = d_V \cdot V < 2V \]
	and thus $V > 2C$. Thus $\M(G)$ is divisible by $3^{V-2C}$, so in particular, by $3$.
\end{proof}

It is interesting that this proof is inherently dependent on the sparse nature of the lattice $L$. If we were to consider a denser tetracellular lattice, such as a triangular pyramid constructed out of regular tetrahedrons, most vertices would be contained in three or four cells, and we could not expect $V > 2C$ to hold. Further, the non-planar graphs enumerated in \cite{fisher} are similarly sparse to our lattice $L$.
It may therefore be worth investigating whether further results on sparse non-planar graphs are more attainable.

\medskip

We conclude this paper with a short list of possible directions for future investigation:
\begin{enumerate}
	\item Is there a subgraph of $L$ with a nice formula for its number of matchings?
	\item Is there a family of non-planar tetracellular graphs which would admit a recursion similar to the Reduction Theorem of \cite{ppowers}?
	\item Can we construct an analogue to the tetrahedral spider lemma for other regular polygons, or more generally find divisibility results for lattices constructed from other regular polygons?
\end{enumerate}
\medskip

\section*{Acknowledgements}
We wish to thank Mihai Ciucu for the many helpful conversations throughout the development of the results presented here, and for his valuable comments and suggestions during the process of reviewing and editing this paper.

\medskip

\bibliographystyle{plain}
\bibliography{articlebib}

\end{document}